\documentclass[11pt]{article}
\usepackage[colorlinks]{hyperref}
\usepackage{mathrsfs,amssymb,amsfonts}
\usepackage{amsmath,amssymb,latexsym,color}
\usepackage[mathscr]{eucal}
\usepackage[numbers,sort&compress]{natbib}
\usepackage{graphicx}
\usepackage{lipsum}
\def\thefootnote{\fnsymbol{footnote}}
\newtheorem{thm}{Theorem}[section]

\newtheorem{lemma}[thm]{Lemma}

\newtheorem{example}[thm]{Example}
\newtheorem{problem}[thm]{Problem}
\newtheorem{remark}[thm]{Remark}

\newtheorem{Notation}[thm]{Notation}

\newcommand{\proof}{{\it Proof.\quad}}
\newcommand{\qed}{\hfill\Box\medskip}
\renewcommand{\thefootnote}{\arabic{footnote}}
\newcommand{\mdi}{{\rm mdim}}
\newcommand{\di}{{\rm dim}}
\newcommand{\gt}{{\rm gt}}
\newcommand{\gpack}{{\rm gpack}}
\newcommand{\ds}{{\rm d}}
\begin{document}
%%%%%%%%%%%%%%%%%%%%%%%%%%%%%%%%%%%%%%%%%%%%%%%%%%%%%%%%
\renewcommand{\abovewithdelims}[2]{\genfrac{[}{]}{0pt}{}{#1}{#2}}

\title{\bf On two problems concerning the geodesic transversal number of graphs}

\author{{Shi Chen$^{a}$,~~Xuanlong Ma$^{a,}$\footnote{Corresponding author},~~Tingzeng Wu$^{b,c}$}
\medskip
\\
{\small\em $^{a}$School of Science, Xi'an Petroleum University, Xi'an 710065, China} \medskip
\\
{\small\em $^{b}$School of Mathematics and Statistics,
Qinghai Minzu University,}
\\
{\small\em Xining, Qinghai 810007, China}\medskip
\\
{\small\em $^{c}$Qinghai Institute of Applied Mathematics, Xining, Qinghai 810007, China}
}
\date{}
\maketitle
%%%%%%%%%%%%%%%%%%%%%%%%%%%%%%%%%%%%%%%%%%%%%%%%%%%%%%%%
\newcommand\blfootnote[1]{%
\begingroup
\renewcommand\thefootnote{}\footnote{#1}%
\addtocounter{footnote}{-1}%
\endgroup
}
%%%%%%%%%%%%%%%%%%%%%%%%%%%%%%%%%%%%%%%%%%%%%%%%%%%%%%%%
\begin{abstract}
We establish sharp upper and lower bounds on the geodesic transversal number of hypercubes. We also obtain a lower bound on the geodesic transversal number of Hamming graphs.
As applications, we provide negative answers to two open problems concerning the geodesic transversal number.

\medskip
\noindent {\em Key words:} Geodesic transversal number; Mixed metric dimension; Geodesic packing number

\medskip
\noindent {\em 2020 MSC:} 05C12; 05C38
\end{abstract}
%%%%%%%%%%%%%%%%%%%%%%%%%%%%%%%%%%%%%%%%%%%%%%%%%%%%%%%%
\blfootnote{E-mail addresses: nuoyoushi@gmail.com (Shi Chen), xuanlma@xsyu.edu.cn (Xuanlong Ma), mathtzwu@163.com (Tingzeng Wu)}
%%%%%%%%%%%%%%%%%%%%%%%%%%%%%%%%%%%%%%%%%%%%%%%%%%%%%%%%
\section{Introduction}
%%%%%%%%%%%%%%%%%%%%%%%%%%%%%%%%%%%%%%%%%%%%%%%%%%%%%%%%
Every graph considered in this paper is finite, simple, and connected. Let $G$ be a graph.
We always use $V(G)$ and $E(G)$ to denote the vertex set and edge set of $G$, respectively.
Given two vertices $u$ and $v$ in $G$, the distance between
$u$ and $v$ is denoted by $\ds_G(u,v)$, or simply by $\ds(u,v)$ when no confusion can arise.
For two distinct vertices $x,y\in V(G)$, we say a vertex $v\in V(G)$
{\em distinguishes} $x$ and $y$ whenever
$
\ds(v,x)\ne \ds(v,y).
$
A subset $S\subseteq V(G)$ is called a {\em metric generator} of $G$ if for every
pair of distinct vertices $x,y\in V(G)$, there exists a vertex in $S$ such that it distinguishes $x$ and $y$.
A metric generator of $G$ with minimum cardinality is called
a {\em metric basis}
of $G$, and the cardinality of a metric basis is referred to as the {\em metric dimension} of $G$.
In the 1970s, the concept of a metric dimension was introduced independently by Harary and Melter \cite{Ham}, and Slater \cite{Sl}.
The metric dimension of a graph has a wide range of applications, including robot navigation in networks, chemistry, pattern recognition and image processing, and so on (see \cite{Ti}).

Given a vertex $v\in V(G)$ and an edge $e=\{u,w\}\in E(G)$, the {\it distance} $\ds(v,e)$ between $v$ and $e$ is defined by
$
\ds(v,e)=\min\{\ds(v,u),\ds(v,w)\}.
$
Let
$
O(G)=V(G)\cup E(G).
$
We say a vertex $v$ of $G$ {\it distinguishes} two distinct elements $x,y\in O(G)$ (vertices or edges) if
$
\ds(v,x)\ne \ds(v,y).
$
For a subset $S$ of $V(G)$, we say $S$ is a
{\em mixed metric generator} of $G$ provided that, for every pair of distinct elements $x,y\in O(G)$, there exists a vertex in $S$ such that it distinguishes $x$ and $y$.
The cardinality of a mixed metric generator with minimum size, denoted by $\mdi(G)$, is called the {\em mixed metric dimension} of $G$.
Kelenc et al. \cite{KeK} first introduced the concept of a mixed metric dimension. In \cite{KeK}, it was proved that computing the mixed metric dimension of a general graph is NP-hard.

Let $x,y\in V(G)$.
A {\it geodesic} is a shortest path between a pair of vertices in a graph. In particular, if $P$ is a shortest path from $x$ to $y$,
then $P$ is called a $x-y$ {\it geodesic}.
A geodesic $P$ in a graph is said to be {\it maximal} if
there exists no other geodesic in $G$, which contains $P$ as a sub-path. A subset $S$ of $V(G)$ is a {\it geodesic transversal} of $G$ provided that any maximal geodesic in $G$ contains at least one vertex in $S$. The {\it geodesic transversal number} of $G$, denoted by gt($G$), is
the minimum cardinality of a geodesic transversal of $G$.
The geodesic transversal number of a graph is also known as the {\it maximal shortest paths cover number} of a graph;
see, for example, \cite{Pe}.
The geodesic transversal number of join and lexicographic product of graphs was studied in \cite{PeS}
A vertex subset $S$ is called a {\it \gt-set} of $G$ if $S$ is a geodesic transversal of
$G$ such that $|S|=\gt(G)$. The
geodesic-transversal problem is NP-complete \cite{Man}.
A {\it geodesic packing} of $G$ is a set of vertex-disjoint maximal geodesics. The maximum
cardinality of a geodesic packing of $G$ is called the {\it geodesic packing number} of $G$, and is denoted by $\gpack(G)$.
In fact, the geodesic transversal problem and the geodesic packing problem are dual min-max invariant problems.
It is proved in \cite{Mang} that the
decision version of the geodesic packing number is NP-complete.

In \cite{Hua}, the authors obtained some results on the mixed metric dimension and geodesic transversal number of a graph.
They showed that $\gt(G) \le \mdi(G)$ holds for any tree, unicyclic graph, and block graph. Thus, they posed the following problem.

\begin{problem}{\rm (\cite[Problem 5.1]{Hua})}\label{ques-1}
Does $\gt(G) \le \mdi(G)$ hold for every connected graph $G$?
\end{problem}

In \cite{Mang}, it was proved that $\gpack(G)\le \gt(G)$ for any graph $G$, and the upper bound is sharp. Moreover,  $\gt(G)$ can be arbitrarily bigger than $\gpack(G)$; see, for example, $\gt(K_n)=n-1$ and $\gpack(K_n)=\lfloor \frac{n}{2}\rfloor$.
Clearly, in $K_{n,n}$ ($n\ge 2$), any maximal geodesic has length $2$. Thus, $\gpack(K_{n,n})=\lfloor \frac{2n}{3}\rfloor$ and $\gt(K_{n,n})=n$. As a result,
$\gt(K_{n,n})/\gpack(K_{n,n})$ is less than or equal to an absolute constant.
For any graph $G$, the following problem was posed.

\begin{problem}{\rm \cite[Problem 2.5]{Mang}}\label{ques-2}
Is there an absolute constant $C$ such that
$\frac{\gt(G)}{\gpack(G)}\le C$, for all graphs $G$?
\end{problem}

In this paper, we establish sharp upper and lower bounds on the geodesic transversal number of hypercubes. We also obtain a lower bound on the geodesic transversal number of Hamming graphs.
As applications, we give negative answers to Problems~\ref{ques-1} and \ref{ques-2}.
%%%%%%%%%%%%%%%%%%%%%%%%%%%%%%%%%%%%%%%%%%%%%%%%%%%%%%%%
\section{Preliminaries}

The \(n\)-dimensional hypercube \(Q_n\) has vertex set
\[
V(Q_n)=\{0,1\}^n=\{(x_1,x_2,\ldots,x_n):x_i\in\{0,1\}
\text{ for any } 1\le i\le n\}.
\]
Namely, every vertex of \(Q_n\) is an \(n\)-tuple whose coordinate is  either \(0\) or \(1\). Two vertices in \(Q_n\) are adjacent if and only if they differ in exactly one coordinate.
Let $x=(x_1,x_2,\ldots,x_n)\in V(Q_n)$.
The complement of $x$ is denoted by $\overline{x}$, that is  $$\overline{x}=(1-x_1,1-x_2,\ldots,1-x_n).$$
Moreover, we write $|x|=\sum_{i=1}^nx_i$, that is, $|x|$ denotes the number of $1$-entries in the vector $x$. For two vertices $x,y$ of $V(Q_n)$, let $x\oplus y$ denote their bitwise addition modulo \(2\). For example, if $x=(1,0,1,0)$ and $y=(1,1,1,0)$, then
$x\oplus y=(0,1,0,0)$. For any $S\subseteq V(Q_n)$ and $x\in V(Q_n)$, define
$$
S\oplus x=\{s\oplus x: s\in S\}.
$$
It is easy to see that $|S\oplus x|=|S|$.

By the basic properties of $Q_n$, the following is valid.

\begin{remark}\label{rem-1}
A path $P=(a,a_1,\ldots,a_n)$ of \(Q_n\) is a maximal geodesic if and only if $\overline{a}=a_n$, and each coordinate changes exactly once along $P$. More precisely, let \(x=(x_1,\ldots,x_n)\in V(Q_n)\), and let \(\pi\) be a permutation of \(\{1,\ldots,n\}\). Define \[
x^{(0)}=x,~~
x^{(k)}
=x\oplus e_{\pi(1)}\oplus\cdots
\oplus e_{\pi(k)},~~1\leq k\leq n,
\]
where \(e_i\) is the \(i\)-th standard unit vector.
We have that
\[
P=(x^{(0)},x^{(1)},\ldots,x^{(n)})
\]
is a maximal geodesic of $Q_n$.
\end{remark}

For \(k\ge 1\) and \(n\ge 2\), Hamming graph \(H(k,n)\) is the graph with vertex set
$$\{(x_1,x_2,\ldots,x_k): x_i\in \{0,1,2,\ldots,n-1\}\} \text{ for each $1\le i \le k$}\},$$
in which two vertices are adjacent if and only if they differ in exactly one coordinate.
In general, we denote $V(H(k,n))$ by \(\{0,1,\ldots,n-1\}^{k}\). Thus, $H(k,n)$
has \(n^k\) vertices and is \(k(n-1)\)-regular. The distance between two vertices equals the number of coordinates in which they differ. Consequently, \(H(k,n)\) has diameter \(k\). Moreover, every maximal geodesic of \(H(k,n)\) has length \(k\) and contains \(k+1\) vertices. In particular, \(H(k,2)\) is the \(k\)-dimensional hypercube \(Q_k\).

The following result will be used in the proof of the subsequent  theorems.

\begin{lemma}\label{nlem-0}
Let \(A_0,A_1,\ldots,A_m\) be $m+1$ finite sets, and let
$
U=\bigcup_{j=0}^{m}A_j.
$
For any \(x\in U\), let
$
I_x=\{j\in\{0,1,\ldots,m\}:x\in A_j\},
$
and let \(r_x\) be the number of maximal intervals of consecutive integers in \(I_x\). Then
\[
\sum_{j=0}^{m}|A_j|-|U|-\sum_{j=1}^{m}|A_{j-1}\cap A_j|
=\sum_{x\in U}(r_x-1).
\]
In particular, we have $$|U|+\sum_{j=1}^{m}|A_{j-1}\cap A_j|\le \sum_{j=0}^{m}|A_j|.$$
\end{lemma}
\proof
For any \(x\in U\), we have $r_x\ge 1$. Set \(k_x=|I_x|\). Clearly, \(k_x\geq 1\), and
$
\sum_{j=0}^{m}|A_j|=\sum_{x\in U}k_x.
$
Now, consider the sum
$
\sum_{j=1}^{m}|A_{j-1}\cap A_j|.
$
For fixed \(x\in U\), its contribution to this sum equals the number of indices \(j\) such that both \(j-1\) and \(j\) belong to \(I_x\).
Since \(I_x\) consists of \(r_x\) maximal intervals of consecutive integers and contains \(k_x\) integers in total, the number of such adjacent pairs is \(k_x-r_x\). Indeed, an interval containing \(\ell\) integers contributes exactly \(\ell-1\) adjacent pairs. Summing over all \(r_x\) intervals gives \(k_x-r_x\). As a result, we have that
\[
\sum_{j=1}^{m}|A_{j-1}\cap A_j|
=\sum_{x\in U}(k_x-r_x).
\]
It follows that
\[
\sum_{j=0}^{m}|A_j|-|U|
-\sum_{j=1}^{m}|A_{j-1}\cap A_j|=\sum_{x\in U}k_x-\sum_{x\in U}1
-\sum_{x\in U}(k_x-r_x)
=\sum_{x\in U}(r_x-1),
\]
as desired. $\qed$

\section{The geodesic transversal number of hypercubes}
The main result of this section is the following theorem:
\begin{thm}\label{qn-gt}
For any $n\ge 2$, we have
$$
\left\lceil \frac{2^n}{\lceil n/2\rceil+1} \right\rceil\le \gt(Q_n)\le\binom{n}{\lfloor n/2\rfloor}.
$$
\end{thm}

As an application of the above result, we will give a negative answer to Problem~\ref{ques-1}. We begin by proving the following lemma.

\begin{lemma}\label{lem-1}
For any $n\ge 1$, we have that
$\gt(Q_n)\le \binom{n}{\lfloor n/2\rfloor}$.
\end{lemma}
\proof
Let
$$
L=\left\{x\in V(Q_n): |x|=\left\lfloor \frac{n}{2}\right\rfloor\right\}.
$$
Then $|L|=\binom{n}{\lfloor n/2\rfloor}$.
It suffices to prove that $L$ is a geodesic transversal. Now let
$$P=(v_0, v_1, \ldots, v_n)$$
be a maximal geodesic of $Q_n$. Thus, by Remark~\ref{rem-1}, we see that  $\overline{v_n}=v_0$. Set $|v_0|=r$. Then $|v_n|=n-r$.
Now without loss of generality, assume that $r\le n-r$. Then
$$
|v_0|\le \frac{n}{2},~~|v_n|\ge \frac{n}{2}.
$$
Note that, by Remark~\ref{rem-1}, each coordinate changes exactly once along $P$. It follows that
$$
\big||v_{i}|-|v_{i-1}|\big|=1,~~ 1\le i\le n.
$$
As a result, there exists at least one index $l$ in $\{0,1,2,\ldots,n\}$ such that $|v_l|=\lfloor n/2\rfloor$. It follows that $v_l\in L$, and so $L$ is a geodesic transversal, as desired.
$\qed$

\begin{lemma}\label{lem-2}
For any $n\ge 2$, we have that
$\gt(Q_n)\ge \left\lceil \frac{2^n}{\lceil n/2\rceil+1} \right\rceil$.
\end{lemma}
\proof
Suppose that $S$ is a $\gt$-set of $Q_n$. Let $x\in V(Q_n)$. In the following, we consider two cases for $n$.

\smallskip
{\bf Case 1.} $n=2m$, where $m\ge 1$.
\smallskip

Set
\begin{equation}\label{meq-1}
q_0=0,~~q_j=e_1\oplus e_2 \oplus \dots \oplus e_{2j},~~1\le j \le m,
\end{equation}
where \(e_i\) is the \(i\)-th standard unit vector.
For any $1\le j \le m$, let
\begin{equation}\label{meq-2}
y_i\in T_i:=\{x\oplus q_{j-1}\oplus e_{2j},x\oplus q_{j-1}\oplus e_{2j-1}\}.
\end{equation}
Then by Remark~\ref{rem-1}, we see that
\begin{equation}\label{eq-1}
(x,y_1,x\oplus q_1,y_2,x\oplus q_2,\ldots,y_m,x\oplus q_m)
\end{equation}
is a maximal geodesic. If $x\oplus q_i\notin S$ for each $0\le i \le m$, and
$T_i\nsubseteq S$ for each $1\le j \le m$, then by \eqref{eq-1},
we can choose a maximal geodesic from $x$ to $x\oplus q_m$ such that every vertex in this maximal geodesic does not belong to $S$, which contradicts the definition of $S$.
We conclude that
\begin{equation}\label{eq-2}
x\oplus q_i\in S \text{ for some } 0\le i \le m
\end{equation}
or
\begin{equation}\label{eq-3}
T_j\subseteq S \text{ for some } 1\le j \le m.
\end{equation}
For any $y\in V(Q_n)$, write
$$
A_i=S\oplus q_i,~~0\le i \le m.
$$
Note that \(-1\equiv1\pmod{2}\). Thus, \eqref{eq-2} is equivalent to
$$
x\in \bigcup_{i=0}^m A_i.
$$
Furthermore, since $q_j=q_{j-1}\oplus e_{2j-1}\oplus e_{2j}$ for each $1\le j \le m$, we deduce that \eqref{eq-3} is equivalent to
$$
x\in (A_{j-1}\cap A_j)\oplus e_{2j-1} \text{ for some } 1\le j \le m.
$$
Now, in either \eqref{eq-2} or \eqref{eq-3}, it follows that
$$
V(Q_n)\subseteq \left(\bigcup_{j=0}^mA_j\right)\cup \left(\bigcup_{i=1}^m((A_{i-1}\cap A_i)\oplus e_{2i-1})\right),
$$
and hence, we have
$$
2^{2m}\le \left|\bigcup_{j=0}^mA_j\right|+\sum_{i=1}^m|A_{i-1}\cap A_i|.
$$
In view of Lemma~\ref{nlem-0}, we deduce that
$
2^{2m}\le \sum_{j=0}^{m}|A_j|=(m+1)|S|,
$
which implies that
$
|S|\ge \frac{2^{2m}}{m+1}.
$
In other words, we obtain
\begin{equation}\label{eq-4}
\gt(Q_{2m})\ge \left\lceil\frac{2^{2m}}{m+1}\right\rceil.
\end{equation}

\smallskip
{\bf Case 2.} $n=2m+1$, where $m\ge 1$.
\smallskip

Assume first \eqref{meq-1} and \eqref{meq-2}.
Then, according to Remark~\ref{rem-1}, we infer that
$$
(x,y_1,x\oplus q_1,y_2,x\oplus q_2,\ldots,y_m,x\oplus q_m,x\oplus q_m\oplus e_{2m+1})
$$
is a maximal geodesic. Let
$$
A_i=S\oplus q_i, ~~ A_{m+1}=S\oplus\sum_{j=1}^ne_j,~~0\le i \le m.
$$
Then, similar to the proof of Case 1, we conclude that
\begin{equation}\label{eq-5}
2^{2m+1}\le \left|\bigcup_{j=0}^{m+1}A_j\right|+\sum_{i=1}^m|A_{i-1}\cap A_i|.
\end{equation}
Combining \eqref{eq-5} and Lemma~\ref{nlem-0}, we obtain
$$
2^{2m+1}\le \left|\bigcup_{j=0}^{m+1}A_j\right|+\sum_{i=1}^m|A_{i-1}\cap A_i|+|A_{m}\cap A_{m+1}|\le \sum_{j=0}^{m+1}|A_j|=(m+2)|S|.
$$
As a result,
\begin{equation}\label{eq-6}
\gt(Q_{2m+1})\ge \left\lceil\frac{2^{2m+1}}{m+2}\right\rceil.
\end{equation}

Finally, combining \eqref{eq-4} and \eqref{eq-6}, we have
$\gt(Q_n)\ge \left\lceil \frac{2^n}{\lceil n/2\rceil+1} \right\rceil$, as desired.
$\qed$

Now, Theorem~\ref{qn-gt} follows from Lemmas~\ref{lem-1} and \ref{lem-2}.

\begin{remark}
It is straightforward to verify that $\gt(Q_2)=2$ and $\gt(Q_3)=3$. Thus, both the upper and lower bounds in Theorem~\ref{qn-gt} are sharp. In fact, by Theorem~\ref{qn-gt}, we also can obtain that $\gt(Q_2)=2$,  $\gt(Q_3)=3$, and $\gt(Q_4)=6$.
\end{remark}

The following theorem indicates that the geodesic transversal number and the mixed metric dimension can be arbitrarily separated in either direction over the class of connected graphs.

\begin{thm}\label{thm-b}
Let \(\mathcal G\) denote the class of all finite connected graphs with at least two vertices. Then there exist connected graphs $G,H$ such that
\begin{equation}\label{eq-7}
\sup_{G\in\mathcal G}
\frac{\operatorname{gt}(G)}{\operatorname{mdim}(G)}
=\infty,
\qquad
\sup_{H\in\mathcal G}
\frac{\operatorname{mdim}(H)}{\operatorname{gt}(H)}
=\infty.
\end{equation}
\end{thm}
\proof
Consider first $Q_n$ for $n\ge 3$. Then, it follows from
\cite[Corollary 2.8]{KeT} that $\mdi(Q_n)=\di(Q_n)$.
Thus, by \cite[Sect. 2]{Ca}, we obtain
\begin{equation}\label{eq-11}
\mdi(Q_n)=\di(Q_n)\le n.
\end{equation}
Moreover, in view of Theorem~\ref{qn-gt}, we have $\gt(Q_n)\ge \frac{2^n}{n+1}$.
This forces that
\begin{equation}\label{eq-8}
\frac{\operatorname{gt}(Q_n)}{\operatorname{mdim}(Q_n)}
\ge \frac{2^n}{n(n+1)}\longrightarrow \infty.
\end{equation}
Next we consider $K_{1,n}$. By \cite[Proposition 3.2]{Man}, we have $\gt(K_{1,n})=1$.
Furthermore, it follows from \cite[Theorem 4.3]{KeK} that $\mdi(K_{1,n})=n$. Thus, we deduce
$$
\frac{\mdi(K_{1,n})}{\gt(K_{1,n})}
=n \longrightarrow \infty.
$$
We conclude that \eqref{eq-7} holds, as desired.
$\qed$

Now, we give a negative answer to Problem~\ref{ques-1}.
By Theorem~\ref{thm-b}, we have $\gt(Q_n) > \mdi(Q_n)$ for all sufficiently large $n$.
In fact, by Theorem~\ref{qn-gt} and \eqref{eq-11}, we have
$$
\mdi(Q_n)\le n< \gt(Q_n),~~n\ge 4,
$$
as desired.

\section{The geodesic transversal
number of Hamming graphs}

For convenience, in the following, let
$$
G_{k,n}=H(k,n),~~[n]=\{0,1,2,\ldots,n-1\},~~V(G_{k,n})=[n]^k.
$$
In this section, we mainly prove the following theorem:

\begin{thm}\label{gt-ham}
For integers \(k\ge 1\) and \(n\ge 2\), we have
\begin{equation}\label{2-eq-10}
\gt(G_{k,n})\ge n^k-O_k\cdot n^{k-2^{1-k}},
\end{equation}
where $O_k$ is a positive constant depending only on $k$.
\end{thm}

Using Theorem~\ref{gt-ham}, we will give a negative answer to Problem~\ref{ques-2}.
We first prove Theorem~\ref{gt-ham}.
Let $S$ be a $\gt$-set of $G_{k,n}$, and write $$A:=V(G_{k,n})\setminus S.$$
Then no maximal geodesic is entirely contained in \(A\).
A \(k\)-dimensional binary coordinate block in \(G_{k,n}\) is a vertex set of the form
\begin{equation*}\label{2-eq-2}
\{a_1,b_1\}\times\{a_2,b_2\}\times\cdots\times\{a_k,b_k\}
\end{equation*}
where \(a_i,b_i\in [n]\) and \(a_i\ne b_i\) for every \(i\in\{1,\ldots,k\}\). Clearly, such a block contains \(2^k\) vertices and induces a subgraph isomorphic to the hypercube \(Q_k\).

\begin{lemma}\label{2-lem-1}
\(A\) contains no any \(k\)-dimensional binary coordinate block.
\end{lemma}
\proof
Suppose, for a contradiction, that \(A\) contains a \(k\)-dimensional binary coordinate block, say  $\{a_1,b_1\}\times\{a_2,b_2\}\times\cdots\times\{a_k,b_k\}$, where \(a_i,b_i\in [n]\) and \(a_i\ne b_i\) for every \(i\in\{1,\ldots,k\}\). Then, by changing the coordinates one at a time, in any order, we can obtain a maximal geodesic of length $k$ from $(a_1,a_2,\ldots,a_k)$ to $(b_1,b_2,\ldots,b_k)$. Clearly, every vertex in this maximal geodesic belongs to $A$, which contradicts the definition of  $A$.
$\qed$

\begin{lemma}\label{2-lem-2}
Let \(f_k(n)\) denote the maximum cardinality of a subset in  \([n]^k\) that contains no any \(k\)-dimensional binary coordinate block. Then, for any fixed integer $k\ge 1$, we have
\begin{equation}\label{2-eq-3}
f_k(n)= O_k\cdot n^{k-2^{1-k}},
\end{equation}
where $O_k$ is a positive constant depending only on $k$.
\end{lemma}
\proof
We proceed by induction on $k$.
If $k=1$, then clearly, $f_1(n)=1$, and so \eqref{2-eq-3} holds whenever $O_1=1$, as desired.

In the following, suppose that $k\ge 2$ and \eqref{2-eq-3} holds for $k-1$. Let $A$ be a subset of $[n]^k$ that contains no any \(k\)-dimensional binary coordinate block. For any $i\in [n]$, define
$$
A_i=\{x\in [n]^{k-1}: (i,x)\in A\}.
$$
Furthermore, for any $x\in [n]^{k-1}$, define
$$
d(x)=|\{i\in [n]: x\in A_i\}|.
$$
Clearly, we have
\begin{equation}\label{2-eq-4}
\sum_{x\in [n]^{k-1}}d(x)=|A|.
\end{equation}

For distinct $i,j\in [n]$, we claim that $A_i\cap A_j$
contains no any binary coordinate block in $[n]^{k-1}$.
In fact, if $A_i\cap A_j$ contains a $(k-1)$-dimensional binary coordinate block, then we can extend it to a \(k\)-dimensional binary coordinate block by adjoining \(\{i,j\}\) as a new coordinate factor, which is impossible because this new \(k\)-dimensional binary coordinate block is contained in $A$. Thus, by the induction hypothesis,
there exists a positive constant $O_{k-1}$ depending only on $k$ such that
\begin{equation}\label{2-eq-5}
|A_i\cap A_j|\le O_{k-1}\cdot n^{k-1-2^{2-k}}.
\end{equation}
Also, for every $x\in [n]^{k-1}$, considering the number of
all pairs of sets \(A_i\) and \(A_j\) so that $x\in A_i\cap A_j$, we obtain
\begin{equation}\label{2-eq-6}
\sum_{1\le i < j \le n}|A_i\cap A_j|
=\sum_{x\in [n]^{k-1}}\binom{d(x)}{2}.
\end{equation}
Now by \eqref{2-eq-4} and the Cauchy-Schwarz inequality, we deduce
$$
\sum_{x\in [n]^{k-1}}d(x)^2\ge \frac{\big(\sum_{x\in [n]^{k-1}}d(x)\big)^2}{n^{k-1}}=\frac{|A|^2}{n^{k-1}}.
$$
Note that $\binom{m}{2}=\frac{m^2-m}{2}$ for every positive integer $m$. It follows that
\begin{equation}\label{2-eq-7}
\sum_{x\in [n]^{k-1}}\binom{d(x)}{2}
=\frac{1}{2}\left(\sum_{x\in [n]^{k-1}}d(x)^2-\sum_{x\in [n]^{k-1}}d(x)\right)\ge \frac{1}{2}\left(\frac{|A|^2}{n^{k-1}}-|A| \right).
\end{equation}
Moreover, by \eqref{2-eq-5}, \eqref{2-eq-6}, and $\binom{n}{2}\le \frac{n^2}{2}$, we conclude
\begin{equation}\label{2-eq-8}
\sum_{x\in [n]^{k-1}}\binom{d(x)}{2}\le
\binom{n}{2}O_{k-1}n^{k-1-2^{2-k}}\le
\frac{O_{k-1}}{2}n^{k+1-2^{2-k}}.
\end{equation}
Combining \eqref{2-eq-7} and \eqref{2-eq-8}, we deduce
\begin{equation}\label{2-eq-9}
|A|^2\le |A|n^{k-1}+O_{k-1}n^{2k-2^{2-k}}.
\end{equation}

If $|A|\le 2n^{k-1}$, then
$
|A|\le 2n^{k-2^{1-k}}
$
by $k\ge 2$, as desired.
Finally, we may assume that $|A|> 2n^{k-1}$. Then $|A|n^{k-1}<|A|^2/2$. It follows from \eqref{2-eq-9} that $$|A|^2/2\le O_{k-1}n^{2k-2^{2-k}},$$
which is equivalent to
$$
|A|\le \sqrt{2O_{k-1}}n^{k-2^{1-k}}.
$$
This proof is now complete.
$\qed$

Now, by Lemmas~\ref{2-lem-1} and \ref{2-lem-2}, we have
$
|A|\le O_k\cdot n^{k-2^{1-k}}.
$
Thus, Theorem~\ref{gt-ham} follows from the fact $|S|=n^k-|A|$.

Next, we give a negative answer to Problem~\ref{ques-2}.
We first consider $\gpack(G_{k,n})$. It is clear that
the number of vertices in every geodesic packing of $G_{k,n}$ must be a multiple of $k+1$. This gives an upper bound for $\gpack(G_{k,n})$, that is
\begin{equation}\label{2-eq-1}
\gpack(G_{k,n})\le \left\lfloor \frac{n^k}{k+1} \right\rfloor \le \frac{n^k}{k+1}.
\end{equation}
Now, combining \eqref{2-eq-1} and \eqref{2-eq-10}, we deduce
$$
\frac{\gt(G_{k,n})}{\gpack(G_{k,n})}\ge
\frac{n^k-O_k\cdot n^{k-2^{1-k}}}{n^k/(k+1)}=
(k+1)(1-O_k\cdot n^{-2^{1-k}})\longrightarrow k+1.
$$
Thus, for a absolute constant $C$, we may take positive integer $k$ with $k+1>C$. Then, for sufficiently large \(n\), we have
$$
\frac{\gt(G_{k,n})}{\gpack(G_{k,n})}>C.
$$
Namely, there exists a graph $G$ such that
$$
\sup_{G\in\mathcal G}
\frac{\gt(G)}{\gpack(G)}
=\infty,
$$
where \(\mathcal G\) denotes the class of all graphs,
which gives a negative answer to Problem~\ref{ques-2}.
\section*{Conflict of interest}
%%%%%%%%%%%%%%%%%%%%%%%%%%%%%%%%%%%%%%%%%%%%%%%%%%%%%%%%
The authors state no conflict of interest.
%%%%%%%%%%%%%%%%%%%%%%%%%%%%%%%%%%%%%%%%%%%%%%%%%%%%%%%%
%%%%%%%%%%%%%%%%%%%%%%%%%%%%%%%%%%%%%%%%%%%%%%%%%%%%%%%%
%\section*{Acknowledgments}
%We are grateful to the referees for many detailed comments and %thoughtful suggestions that have helped improve this paper substantially.

%%%%%%%%%%%%%%%%%%%%%%%%%%%%%%%%%%%%%%%%%%%%%%%%%%%%%%%%

\begin{thebibliography}{99}
%%%%%%%%%%%%%%%%%%%%%%%%%%%%%%%%%%%%%%%%%%%%%%%%%%%%%%%
\bibitem{Ca} J. C\'aceres, C. Hernando, M. Mora, I. M. Pelayo, M. L. Puertas, C. Seara, D. R. Wood,
On the metric dimension of Cartesian products of graphs,
{\it SIAM J. Discrete Math.}, \textbf{21} (2007), 423--441.


\bibitem{Ham} F. Harary, R. A. Melter, On the metric dimension of a graph, {\it Ars Combin.}, \textbf{2} (1976), 191--195.

\bibitem{Hua} H. Hua, Y. Chen,
X. Hua, Further results on the mixed metric dimension of graphs, {\it Discrete Appl. Math.}, \textbf{367} (2025),  99--106.

\bibitem{KeK} A. Kelenc, D. Kuziak, A. Taranenko, I. G. Yero, Mixed metric dimension of graphs, {\it Appl. Math. Comput.} \textbf{314} (2017), 429--438.

\bibitem{KeT} A. Kelenc, A. T. M. Toshi, R. \v{S}krekovski,
I. G. Yero, On metric dimensions of hypercubes, {\it Ars Math. Contemp.}, \textbf{23} (2023), \#P2.08.

\bibitem{Man} P. Manuel, B. Bre\v{s}ar, S. Klav\v{z}ar,
The geodesic-transversal problem,
{\it Appl. Math. Comput.}, \textbf{413} (2022), 126621.

\bibitem{Mang} P. Manuel, B. Bre\v{s}ar, S. Klav\v{z}ar,
Geodesic packing in graphs,
{\it Appl. Math. Comput.}, \textbf{459} (2023), 128277.

\bibitem{Pe} I. Peterin, G. Semani\v{s}in, On maximal shortest paths cover, {\it Mathematics}, \textbf{9} (2021), 1592.

\bibitem{PeS} I. Peterin, G. Semani\v{s}in, Geodesic transversal problem for join and lexicographic product of graphs, {\it Comput. Appl. Math.}, \textbf{41} (2022), 128.

\bibitem{Sl} P. J. Slater, Leaves of trees, {\it Congr. Numer.}, \textbf{14} (1975), 549--559.


\bibitem{Ti} R. C. Tillquist, R. M. Frongillo, M. E. Lladser, Getting the lay of the land in discrete space: a survey of metric dimension and its applications,
    {\it SIAM Rev.}, \textbf{65} (2023), 919--962.


%%%%%%%%%%%%%%%%%%%%%%%%%%%%%%%%%%%%%%%%%%%
%%%%%%%%%%%%%%%%%%%%%%%%%%%%%%%%%%%%%%%%%%
%%%%%%%%%%%%%%%%%%%%%%%%%%%%%%%%%%%%%%%%%%
%%%%%%%%%%%%%%%%%%%%%%%%%%%%%%%%%%%%%%%%%%
%%%%%%%%%%%%%%%%%%%%%%%%%%%%%%%%%%%%%%%%
\end{thebibliography}
\end{document}